\documentclass[11pt]{amsart}    
\usepackage{graphicx}
\newcommand{\cev}[1]{\reflectbox{\ensuremath{\vec{\reflectbox{\ensuremath{#1}}}}}}

\usepackage[pagebackref=true,colorlinks,linkcolor=blue,citecolor=magenta]{hyperref}
\usepackage{amsmath,amssymb,latexsym} 
\usepackage{graphicx}\usepackage{centernot}

\usepackage{fullpage}
\usepackage[square, numbers]{natbib}   
\usepackage{comment}

\usepackage{enumitem}

\newtheorem{theorem}{Theorem} 

\newtheorem{alphtheorem}{Theorem}

\newtheorem{alphlemma}{Lemma}

\newtheorem{lemma}{Lemma}
\theoremstyle{definition}
\newtheorem{definition}{Definition}

\theoremstyle{remark}

\def\isdef{\mbox {$\ \stackrel{\rm def}{=} \ $}}
\def\zero{\boldsymbol{0}}

\def\ds{\displaystyle}
\def\KG{\operatorname{KG}}
\def\SG{\operatorname{SG}}

\def\alt{\operatorname{alt}}
\def\salt{\operatorname{salt}}
\def\HH{\mathcal{H}}

\def\mod{\operatorname{mod}}

\title{ Circular chromatic number of induced subgraphs of  Kneser graphs}
\author{Meysam Alishahi}
\address{M. Alishahi, 
School of Mathematical Sciences,
Shahrood University of Technology, Shahrood, Iran}
\email{meysam\_alishahi@shahroodut.ac.ir}

\author{Ali Taherkhani}
\address{A. Taherkhani, 
Department of Mathematics, Institute for Advanced Studies in Basic Sciences (IASBS), Zanjan 45137-66731, Iran}
\email{ali.taherkhani@iasbs.ac.ir}

\begin{document}
\maketitle

\begin{abstract}
\noindent 
Investigating the equality of the  chromatic number and 
the circular chromatic number of graphs has been an active stream of research for last decades. 
In this regard, Hajiabolhassan and Zhu~[{\it Circular chromatic number of {K}neser graphs, 
Journal of Combinatorial Theory Series B}, 2003] proved that if $n$ is sufficiently large with respect to 
  $k$, then the Schrijver graph ${\rm SG}(n,k)$ 
has the same chromatic and circular chromatic number.  Later, Meunier~[{\it A topological lower bound for the circular chromatic number of
  {S}chrijver graphs, Journal of Graph Theory}, 2005] and independently, 
   Simonyi and Tardos~[{\it Local chromatic number, {K}y {F}an's theorem and circular colorings, Combinatorica}, 2006] 
 proved that $\chi({\rm SG}(n,k))=\chi_c({\rm SG}(n,k))$ if  $n$ is even. 
In this paper, we study the circular chromatic number of induced subgraphs of Kneser 
graphs. In this regard,  we shall first generalize the preceding result to $s$-stable Kneser 
graphs. Furthermore, as a generalization of  Hajiabolhassan and Zhu's result, we prove 
that if $n$ is large enough with respect to $k$, then any sufficiently large induced subgraph 
of the Kneser graph $\KG(n,k)$ has the same chromatic number and circular chromatic number. 
\\

\end{abstract}

\section{\bf Introduction}
Throughout the paper, the symbol $[n]$ stands for the set $\{1,\ldots, n\}$.
Let $n$ and $d$ be two positive integers. The {\it circular complete graph $K_{n\over d}$}   
is a graph with the vertex set $[n]$ and two vertices $i$ and $j$ are adjacent 
whenever $d\leq |i-j|\leq n-d$. For a graph $G$, the circular chromatic number of $G$, 
denoted $\chi_c(G)$, is defined as follows:
$$\chi_c(G)\isdef\inf\left\{{n\over d}\;:\; \mbox{there is a homomorphism from $G$ to $K_{n\over d}$}\right\}.$$
It is known that the infimum can be replaced by minimum.
Moreover, one can see that $\chi(G)-1<\chi_c(G)\leq\chi(G)$, see~\cite{MR1815614}. 
Therefore, the circular chromatic number is a refinement of the chromatic number.
It is a natural question to ask for which graphs $G$, we have $\chi_c(G)=\chi(G)$.
However, it turns out to be a difficult question.  
Pavol Hell~\cite{HELL199092} Proved that    the problem of determining whether a graph has the circular chromatic number at most $\frac{n}{d} $ is NP-Hard.  Hatami and Tusserkani~\cite{HaTu2004} showed that the problem of determining whether or not $\chi_c(-)=\chi(-)$ is  NP-Hard
even if the chromatic number is known. Therefore, 
finding necessary conditions for graphs to have the same chromatic and circular 
chromatic number turns out to be an interesting problem. This problem has received a 
significant  attention, for instance see~\cite{MR2601263,MR1983360,MR1815614,MR2249284}.

For two positive integers $n$ and $k$, where $n\geq 2k$, the {\it  Kneser graph $\KG(n,k)$} 
is a graph with the vertex set ${[n]\choose k}$, that is, the family of all $k$-subsets of $[n]$,  
and two vertices are adjacent if their corresponding $k$-subsets are disjoint. 
Kneser 1955~\cite{kneser} conjectured that  the chromatic number of $\KG(n,k)$ is
$n-2k+2$. In 1978, Lov{\'a}sz~\cite{MR514625} gave an affirmative answer to Kneser's conjecture.  
He used algebraic topological tools, giving birth to the field of topological combinatorics. 
For a positive integer $s$, a nonempty subset ${\mathcal S}$ of $[n]$ is said to be $s$-stable
if  for any two different elements  $i$ and $j$ in $S$, we have $s\leq |i-j|\leq n-s$. 
Throughout the paper, 
the family of all $s$-stable $k$-subsets of $[n]$ is denoted by ${[n]\choose k}_s$. 
The  subgraph of $\KG(n,k)$ induced by all $s$-stable $k$-subsets of $[n]$ is called 
the {\it  $s$-stable Kneser graph} and is denoted by $\KG_s(n,k)$.
The $2$-stable Kneser graph $\KG_2(n,k)$ is known as the Schrijver graph $\SG(n,k)$. 
Schrijver~\cite{MR512648} proved that Schrijver graphs are vertex critical subgraphs of Kneser graphs with the same chromatic number. 
Meunier~\cite{MR2793613} showed that for any two positive integers 
$n$ and $k$, where $n\geq sk$, the $s$-stable Kneser graph $\KG_s(n,k)$ can be colored by
$n-s(k-1)$ colors and conjectured that the chromatic number is $n-s(k-1)$. 
Jonsson~\cite{jonsson} proved that this conjecture is true provided that $s\geq 4$ and $n$ 
is sufficiently large with respect to  $k$ and $s$. Also, Chen~\cite{Chen-2015} confirmed the 
Meunier's conjecture for even values of $s$.

Lov{\'a}sz's theorem~\cite{MR514625} has been generalized in several aspects. 
For a hypergraph $\HH$, 
the {\it  general Kneser graph $\KG(\HH)$} is a graph with the vertex set $E(\HH)$ and two 
vertices are adjacent if their corresponding edges are vertex disjoint. 
Dol'nikov~\cite{MR953021} generalized Lov{\'a}sz's result and proved that the chromatic 
number of $\KG(\HH)$ is at least the {\it  colorability defect of $\HH$}, denoted ${\rm cd}(\HH)$, 
where the colorability defect of $\HH$ is the minimum number of vertices should be excluded 
from $\HH$ so that the induced subhypergraph on the remaining vertices is $2$-colorable.

For a vector $X=(x_1,\ldots,x_n)\in\{-,0.+\}^n$, a sequence 
$x_{i_1},x_{i_2},\ldots,x_{i_t}$ ($i_1<\cdots<i_t$) is called an {\it alternating subsequence of $X$ with length $t$} if $x_{i_j}\neq 0$ for each $j\in\{1,\ldots,t\}$
and $x_{i_j}\neq x_{i_{j+1}}$ for each $j\in\{1,\ldots,t-1\}$. The maximum length of an 
alternating subsequence of $X$ is called the {\it  alternation number of $X$}, denoted $\alt(X)$. 
For $\zero\isdef(0,\ldots,0)$, we define $\alt(\zero)\isdef 0$. Aslo, we define $X^+$ and $X^-$ 
to be respectively the indices  of positive and negative coordinates of $X$, i.e.,
$$X^+\isdef\{i\;:\; x_i=+\}\quad\mbox{and}\quad X^-\isdef\{i\;:\; x_i=-\}.$$
Note that both $X^+$ and $X^-$ are subsets of $[n]$ and  by abuse of notation, we can write $X=(X^+,X^-)$. 
For two vectors $X,Y\in \{-,0,+\}^n$, by $X\subseteq Y$, we mean 
$X^+\subseteq Y^+$ and $X^-\subseteq Y^-$. 

Let $\HH=(V,E)$ be a hypergraph and   $\sigma:[n]\longrightarrow V(\HH)$ be a bijection.  
The {\it  alternation number of $\HH$ with respect to $\sigma$}, denoted $\alt_\sigma(\HH)$, 
is the maximum possible value of 
an $\alt(X)$ such that $E(\HH[\sigma(X^+)])=E(\HH[\sigma(X^-)])=\varnothing$. Also, 
the {\it  strong alternation number of $\HH$ with respect to $\sigma$}, 
denoted $\salt_\sigma(\HH)$,  is the maximum possible value of 
an $\alt(X)$ such that $E(\HH[\sigma(X^+)])=\varnothing$ 
or $E(\HH[\sigma(X^-)])=\varnothing$. 
The {\it  alternation number of $\HH$}, denoted $\alt(\HH)$, and the {\it  strong 
alternation number of $\HH$}, denoted $\salt(\HH)$, are respectively the minimum 
values of $\alt_\sigma(\HH)$ and $\salt_\sigma(\HH)$, where the minimum is taken 
over all bijections $\sigma:[n]\longrightarrow V(\HH)$.
The first present author and Hajiabolhassan~\cite{alihajijctb2015} proved 
the following theorem. 
\begin{alphtheorem}\label{alihaji}
For any hypergraph
$\HH$, we have
$$\chi(\KG(\HH))\geq \max\left\{|V(\HH)|-\alt(\HH), |V(\HH)|-\salt(\HH)+1\right\}. $$
\end{alphtheorem}
One can simply see that
this result improves the aforementioned 
Dol'nikov's result~\cite{MR953021}. 
Using this lower bound, the chromatic number of several families of 
graphs are computed, for instance see~\cite{2014arXiv1401.0138A,2014arXiv1403.4404A,2014arXiv1407.8035A,2015arXiv150708456A,AliHajiMeu2016}.

In 1997, Johnson, Holroyd, and Stahl~\cite{MR1475894} proved that 
$\chi_c(\KG(n,k))=\chi(\KG(n,k))$ provided that $2k+1\leq n\leq2k+2$ or $k=2$ . 
They also conjectured that the circular chromatic number of  Kneser graphs is always equal to their chromatic number. 
This conjecture has been studying in several articles. Hajiabolhassan and 
Zhu~\cite{MR1983360} 
using a combinatorial method proved that  if $n$ is large enough with respect to $k$,  
then $\chi_c(\KG(n,k))=\chi(\KG(n,k))$. Later, using algebraic topology,  
Meunier~\cite{MR2197228} 
and Simonyi and Tardos~\cite{MR2279672} independently confirmed this conjecture 
for the case of  even $n$. It should be mentioned that the results 
by Hajiabolhassan and 
Zhu~\cite{MR1983360}, Meunier~\cite{MR2197228}, and Simonyi and 
Tardos~\cite{MR2279672} are indeed proved for the Schrijver graph $\SG(n,k)$. Eventually in 2011, Chen~\cite{MR2763055} confirmed the 
Johnson-Holroyd-Stahl conjecture. Chen's proof was simplified in~\cite{MR2971704}. 
The first present author, Hajiabolhassan, and Meunier~\cite{AliHajiMeu2016} generalized 
Chen's result to a larger family of graphs. They introduced a sufficient 
conditions for a hypergraph $\HH$ having $\chi(\KG(\HH))=\chi_c(\KG(\HH))$.

\subsection*{Plan}
The paper contains two main sections. In Section~\ref{stable}, we shall investigate the coloring properties of stable Kneser graphs. In this regard, we prove the equality of the  chromatic number and the circular chromatic number of $s$-stable Kneser graph $\KG_s(n,k)$ provided that 
$n\geq(s+2)k-2$ and both $n$ and $s$ are even. In the last section, 
we study the circular chromatic number of large 
induced subgraphs of Kneser graphs. Indeed, it is proved that, for large enough $n$, 
any sufficiently large induced subgraph of the Kneser graph $\KG(n,k)$ has the same 
chromatic number and circular chromatic number. In particular, giving a partial answer to a 
question posed by Lih and Liu~\cite{Ko-We-2002}, we present a 
threshold $n(k,s)$ such that for any $n\geq n(k,s)$,
the chromatic number and circular chromatic number of $\KG_s(n,k)$ are equal. 

\section{Chromatic Number of Stable Kneser Graphs}\label{stable}
As it is mentioned in the previous section, the chromatic number of $s$-stable Kneser graph
$\KG_s(n,k)$ is determined provided that $k$ and $s$ are fixed and $n$ is large enough~\cite{jonsson} 
or $s$ is even~\cite{Chen-2015}.
In this section, we first present a generalization of Theorem~\ref{alihaji}. 
Using this generalization, for any even $s$, we prove that any proper coloring of $s$-stable 
Kneser graph $\KG_s(n,k)$ contains a large colorful complete bipartite subgraph, which
immediately implies the chromatic number of $s$-stable Kneser graph $\KG_s(n,k)$. 
Also, this result concludes that the circular chromatic number of  $s$-stable Kneser graph
$\KG_s(n,k)$ equals to their chromatic number provided that $n\geq(s+2)k-2$ and both $n$ and $s$ are even.

Tucker's lemma is a combinatorial counterpart of the
Borsuk-Ulam theorem with several useful applications, for instance, see
\cite{MR1988723,MR2057690}.

\begin{alphlemma}\label{Tucker}
{\rm(Tucker's lemma {\rm \cite{MR0020254} )}} Let 
$\lambda:\{-,0,+\}^n\setminus\{\zero\}\longrightarrow\{\pm1,\ldots,\pm m\}$ 
be a  map satisfying the following properties:
\begin{itemize}
\item it is antipodal: $\lambda(-X)=-\lambda(X)$ for each $X\in\{-,0,+\}^n\setminus\{\zero\}$, and
\item it has no complementary edges: there is no $X$ and $Y$ in $\{-,0,+\}^n\setminus\{\zero\}$ such that 
$X\subseteq Y$ and $\lambda(X)=-\lambda(Y)$. 
\end{itemize}
Then $m\geq n$. 
\end{alphlemma}

There are several results concerning the existence of a large complete bipartite 
multi-colored subgraph in any proper coloring of a graph $G$, 
see~\cite{alihajijctb2015,MR2763055,SiTaZs13,MR2279672,MR2351519}. 
In fact, for a graph $G$, it is shown that 
the order of such subgraph is bounded below by some 
invariants of some topological
spaces defined according to the structure of $G$. 
In what follows, we present a similar type result with a combinatorial proof. 
Note that since there is a purely combinatorial proof~\cite{MR2057690} for 
Tucker's lemma, any proof
using Tucker's lemma with combinatorial approach can be considered as a 
purely combinatorial proof.

Part of the following theorem is proved in~\cite{alihajijctb2015} by use of
Ky-Fan's lemma~\cite{MR0051506}, a generalization of Tucker's lemma. 
Here, we sketch the proof with a different approach and by use of Tucker's lemma instead 
of Ky-Fan's lemma.  
\begin{theorem}\label{lower}
Let $\HH$ be a hypergraph and set 
$t= \max\left\{|V(\HH)|-\alt(\HH), |V(\HH)|-\salt(\HH)+1\right\}$. For any proper coloring 
$c:V(\KG(\HH))\longrightarrow[C]$,  there exists a complete bipartite subgraph 
$K_{\left\lfloor{t\over 2}\right\rfloor,\left\lceil{t\over 2}\right\rceil}$ of $\KG(\HH)$ all vertices of 
which receive different colors and moreover, these different 
colors occur alternating on the two parts 
of the bipartite graph with respect to their natural order. 
\end{theorem}
\begin{proof}
Let $\sigma_1,\sigma_2:[n]\longrightarrow V(\HH)$ be two bijections
for which we have $\alt(\HH)=\alt_{\sigma_1}(\HH)$ and $\salt(\HH)=\salt_{\sigma_2}(\HH)$.
Now, we shall follow the proof with two different cases $t=n-\alt(\HH)$ and $t=n-\salt(\HH)+1$. 
Assume that $t=n-\alt(\HH)$ (resp. $t=n-\salt(\HH)+1$). 
For simplicity of notation, by identifying the set $V(\HH)$ and $[n]$ via the  
bijection $\sigma_1$ (resp. $\sigma_2$),  we may assume that $V(\HH)=[n]$.  

For each $X=(X^+,X^-)\in\{-,0.+\}^n\setminus\{\zero\}$, 
define $c(X)\isdef(c(X^+),c(X^-))\in \{-,0,+\}^C$ to be a signed vector, where 
$$ c(X^+)\isdef\left\{c(e)\;:\; e\in E(\HH)\; \&\; e\subseteq X^+  \right\}\quad \mbox{ and } 
\quad c(X^-)\isdef\left\{c(e)\;:\; e\in E(\HH)\; \&\; e\subseteq X^- \right\}.$$
For each $X\in \{-,0,+\}^n\setminus\{\zero\}$, define  $\lambda(X)$
as follows. 
\begin{itemize}
\item If $\alt(X)\leq \alt_{\sigma_1}(\HH)$ (resp. $\alt(X)\leq \alt_{\sigma_2}(\HH)$),
then define $\lambda(X)=\pm \alt(X)$, where the sign is positive if the first nonzero term 
of $X$ is positive and is negative otherwise. 

\item If $\alt(X)\geq \alt_{\sigma_1}(\HH)+1$ (resp. $\alt(X)\geq \alt_{\sigma_2}(\HH)+1$),
then define $\lambda(X)=\pm (\alt_{\sigma_1}(\HH)+\alt(c(X)))$ 
(resp. $\lambda(X)=\pm (\salt_{\sigma_2}(\HH)+\alt(c(X))-1)$), 
where the sign is positive if the first nonzero term of $c(X)$ is positive and is 
negative otherwise. 
\end{itemize}
One can simply check that $\lambda$ satisfies the conditions of Lemma~\ref{Tucker}. 
Consequently, there should be an $X\in \{-,0,+\}^n\setminus\{\zero\}$ such that 
$|\lambda(X)|=\lambda(X)\geq n$. 
Clearly, we should have $\alt(X)\geq \alt_{\sigma_1}(\HH)+1$ 
(resp. $\alt(X)\geq \salt_{\sigma_2}(\HH)+1$). Therefore, the definition of $\lambda(X)$
 implies that
$\alt(c(X))\geq n-\alt_{\sigma_1}(\HH)$ (resp. $\alt(c(X))\geq n-\salt_{\sigma_2}(\HH)+1$). 
Let $Z=(Z^+,Z^-)\subseteq c(X)$ be a signed vector such that $\alt(Z)=|Z|=\alt(c(X))=t$. 
Note that if $Z^+\cup Z^-=\{i_1,i_2,\ldots,i_t\}$, where $1\leq i_1<\cdots< i_t\leq C$, 
then we should have $Z^+=\{i_{j}\;:\; \mbox{$j\in[t]$ is odd}\}$ and 
$Z^-=\{i_{j}\;:\; \mbox{$j\in[t]$ is even}\}$. 
For an $j\in[t]$, if $j$ is odd (resp. even), then according to the definition of $c(X)$,  
there is 
an edge  $e\in E(\HH)$ such that $e\subseteq X^+$ (resp. $e\subseteq X^-$) with $c(e)=i_j$. 
Note that the induced subgraph of $\KG(\HH)$ on the vertices  $\{e_1,\ldots,e_t\}$ 
contains the desired complete bipartite graph. 
\end{proof}

Let $n, k$, and $s$ be positive integers, where $n\geq sk$ and $s$ is even. 
It is not difficult to see that if $n$ is large enough (with respect to  $s$ and $k$), 
then any $2$-stable $({s\over 2}(k-1)+1)$-subset of $[n]$ contains 
an $s$-stable $k$-subset of $[n]$. In the following two lemmas, 
we shall prove that $n\geq (s+2)k-2$ would be sufficient for this observation. 
\begin{lemma}\label{case2}
Let  $s$ be an even  positive integer and let $n=2s+2$. If ${\mathcal S}$ is a $2$-stable subset   of $[n]$ of  cardinality $\frac{s}{2}+1$, 
then there are $a,a'\in {\mathcal S}$ such that $a-a'\in \{s,s+1,s+2\}.$
\end{lemma}

\begin{proof}
Without loss of generality, we may assume that $1\in {\mathcal S}$ and $2s+2\not\in {\mathcal S}$. If $s+1\in {\mathcal S}$, then there is nothing to prove. 
Therefore, let $s+1\not\in {\mathcal S}.$
For  $1\leq i\leq \frac{s}{2}$, define $B_i=\{2i-1, 2i,2i+s,2i+s+1\}.$ Therefore, for some $i$, $1\leq i\leq \frac{s}{2}$, $|B_i\cap {\mathcal S}|=2$.
 Let $a,a'\in B_i\cap {\mathcal S}$, since $\mathcal S$ is $2$-stable,  we have $a-a'\in\{s,s+1,s+2\}$.
\end{proof}

\begin{lemma}\label{stsable}
Let $k$ and $n$ be two positive integers and let $s$ be  an even positive integer, where  $n\geq(s+2)k-2.$ 
If ${\mathcal S}$ is a $2$-stable subset of $[n]$ of  cardinality $\frac{s}{2}(k-1)+1$, then there is an $s$-stable $k$-subset of ${\mathcal S}$.
In particular, $\salt\left([n],{n\choose k}_{s}\right)=s(k-1)+1$
\end{lemma}\label{casek}
\begin{proof} { 
First note that for given $k$ and $s$, if the statement is true for some $n\geq k(s+2)-2$, then it is true for all integers $n'\geq n.$ Therefore 
it is enough to prove the lemma for $n= k(s+2)-2.$ 

We use induction on $k$ to prove the lemma. The validity of lemma when $k=1$ is trivial and when $k=2$ it was shown  in Lemma~\ref{case2}.
Thus, we may assume that $k\geq 3.$ 

If for each $i\in {\mathcal S}$, we have  
$\{i+s,i+s+1, i+s+2\}\cap{\mathcal S}\neq\varnothing$ (where addition is modulo $n$), 
then we can greedily find  an $s$-stable $k$-subset, 
and there is nothing to prove.
Otherwise, without loss of generality, assume that $n-s-1\in{\mathcal S}$ and $n-1,n,1\not\in{\mathcal S}.$

Set $ A_{_{n-s-1}}=\{n-s-1,n-s,\cdots,n\}$.
Note that since $n-1,n\not\in S$, we have $|A_{_{n-s-1}}\cap {\mathcal S}|=\frac{s}{2}-\beta$, for some $0\leq\beta\leq \frac{s}{2}.$ 
Now, consider $[n]\setminus A_{_{n-s-1}}$ and ${\mathcal S}\setminus A_{_{n-s-1}}$.
Set ${\dot n}= n-(s+2)$ and ${\dot{\mathcal S}}={\mathcal S}\setminus A_{_{n-s-1}}.$
Note that   $[\dot n]$ and $[n]\setminus A_{_{n-s-1}}$ are equal 
and since $1\not\in {\mathcal S}$,   ${\dot{\mathcal S}}$ is  a $2$-stable subset of  $[\dot n]$ 
of cardinality $(k-2)\frac{s}{2}+\beta+1$.

 Define the $s$-subset $B$ of $[\dot n]$ by
  $$B\isdef\{n-2s-1,n-2s,\cdots,n-s-2\}.$$
By induction, we may consider the following two cases:\\
(i) If there is an $s$-stable $(k-1)$-subset of ${\dot{\mathcal S}}$, say $\dot D$, which has no element of $B$, 
then it is readily verified that $D={\dot D}\cup\{n-s-1\}$ is an $s$-stable $k$-subset of $[n]$,
completing the proof in this case.\\
(ii) If there are at least $\beta+1$ $s$-stable $(k-1)$-subsets of ${\dot{\mathcal S}}$, say ${\dot D}_1,{\dot D}_2,\cdots,{\dot D}_{\beta+1}$,
such that each   ${\dot D}_i$ has exactly one distinct  element  of $B$, say $b_i$.

Now, consider the $2$-stable subset $\{b_1,b_2,\cdots,b_{\beta+1}\}\cup({\mathcal S}\cap A_{_{n-s-1}})$, by Lemma~\ref{case2}, there exist  two elements $\cev{b},a$
such that $a-\cev{b}\in \{s,s+1,s+2\}.$ 
Since $n-1,n\not\in {\mathcal S}$, both $\cev{b},a$ are not in $A_{_{n-s-1}}$. Hence, we may assume that 
$a\in A_{_{n-s-1}}$ and $\cev{b}=b_i$ for some $i$, $1\leq i\leq\beta+1$. Let ${ \vec b}$ be the  smallest element of ${\dot D}_i$.
Since ${\dot D}_i$ is an $s$-stable $(k-1)$-subset of $[\dot n]$,  therefore we have 
$s\leq\cev{b}-\vec{b}\leq {\dot n}-s=n-(2s+2)$. On the other hand, $s\leq a-\cev{b}\leq s+2$.
Therefore, $2s\leq a-\vec{b}\leq n-s$. Therefore,  ${\dot D}_i\cup\{a\}$ is an $s$-stable $k$-subset of $[n]$ as diseried.

Note that for an $X\in\{-,0,+\}^n\setminus\{\zero\}$ with $\alt(X)\geq s(k-1)+2$, both $X^+$ and $X^-$
contain $2$-stable subsets of size at least ${s\over 2}(k-1)+1$, which implies that
both $X^+$ and $X^-$
contain $s$-stable subsets of size at least $k$.
This concludes 
that $\salt\left([n],{n\choose k}_{s}\right)=s(k-1)+1$
}\end{proof}
We remind the reader that Meunier~\cite{MR2793613} showed that $\KG_s(n,k)$ has a proper
coloring with $n-s(k-1)$ colors.  Note that if we set $\HH=([n],{[n]\choose k}_s)$, 
then $\KG(\HH)=\KG_s(n,k)$. Clearly, using these observations,  Lemma~\ref{stsable}, and Theorem~\ref{lower}, 
we have the next theorem. 
\begin{theorem}\label{colorfulstable}
Let $k$ and $n$ be two positive integers and let $s$ be  an even positive integer, where  $n\geq(s+2)k-2.$ 
Any proper coloring of $\KG_s(n,k)$ contains a complete bipartite subgraph 
$K_{\left\lfloor {t\over 2}\right\rfloor,\left\lceil {t\over 2}\right\rceil}$, where $t=n-s(k-1)$ such that
all vertices of this subgraph receive different colors and  these different 
colors occur alternating on the two parts 
of the bipartite graph with respect to their natural order. In particular, we have
$\chi(\KG_s(n,k))=n-s(k-1)$. 
\end{theorem}
Let $r$ be a positive integer. For an $r$-coloring $c$ of a given graph $G$, 
a cycle $C=v_1,v_2,\ldots,v_m,v_1$ is called {\it tight} if for each $i\in [m]$, 
we have $c(v_{i+1})=c(v_i)+1(\mod r)$.
It is known~\cite{MR1815614} that $\chi_c(G)=r$ if and only if the graph 
$G$ is $r$-colorable and every $r$-coloring of $G$ contains a tight cycle. 
In view of this result, to prove the next theorem, it suffices to show that any 
proper $(n-s(k-1))$-coloring of $\KG_s(n,k)$ contains a tight cycle. 
\begin{theorem}
Let $n,k$, and $s$  be a  positive integers, where $n$ and $s$ are even and $n\geq(s+2)k-2.$ Then, we have
$$\chi_c(\KG_s(n,k))=n-s(k-1).$$
\end{theorem}
\begin{proof}
In view of the former discussion, to prove the assertion,
let $c$ be a proper $(n-s(k-1))$-coloring of $\KG_s(n,k)$. Consider the complete bipartite 
subgraph $K_{{n-s(k-1)\over 2},{n-s(k-1)\over 2}}$ 
of $\KG_s(n,k)$, whose existence is ensured by Theorem~\ref{colorfulstable}.  
Clearly, this subgraph contains a tight cycle, which completes the proof. 
\end{proof}

Original Lov\'{a}sz's proof~\cite{MR514625}  of Kneser's conjecture is rather long and 
complicated.  B{\'a}r{\'a}ny~\cite{MR514626}, using Gale's lemma~\cite{MR0085552}, 
presented a short proof of this result. 
For $n>2k$, Gale~\cite{MR0085552} proved that 
the set $[n]$ can be identified with a subset of $S^{n-2k}$ in such a way that any open 
hemisphere contains at least one $k$-subset of $[n]$ (a vertex of $\KG(n,k)$).
Schrijver~\cite{MR512648} generalized 
Gale's lemma to $2$-stable $k$-subsets of $[n]$. He also used this generalization 
to prove that $\chi\left(\SG(n,k)\right)=n-2k+2$. For an interesting proof of 
Gale's lemma, see~\cite{MR1829620}. Moreover, the first presented author 
and Hajiabolhassan~\cite{2016arXiv160708780A} generalized Glale's lemma. 
For any hypergraph $\HH=(V,E)$, 
they introduce a lower bound for the maximum possible value of $m$ for which  
there is a subset $X$ of $S^m$ and a suitable identification of $V$ with $X$  
such that any open hemisphere of $S^m$ contains an edge of $\HH$. The next lemma can be obtained 
directly  from this result. However, for the sake of completeness, we prove it here with a 
different approach. 
\begin{lemma}\label{galenew}
Let $k$ and $n$ be two positive integers and let $s$ be  an even positive integer, where  $n\geq(s+2)k-2.$ There exists an $n$-subset $X$ of  $S^{n-s(k-1)-2}$ and a suitable identification between 
$X$ and $[n]$ such that every open hemisphere of $S^{n-s(k-1)-2}$ contains an $s$-stable $k$-subset of $[n]$.
\end{lemma}
\begin{proof}
Set $p={s\over 2}(k-1)+1$. 
In view of the generalization of Gale's lemma by Schrijver~\cite{MR512648}, there exists an 
$n$-subset $X$ of $S^{n-2p}$ and an identification of $X$ with $\{1,\ldots, n\}$ 
such that  any open hemisphere of $S^{n-2p}$ contains a $2$-stable $p$-subset of 
$\{1,\ldots, n\}$. Now, by Lemma \ref{stsable}, any $2$-stable $p$-subset contains an 
$s$-stable $k$-subset. This implies that any open hemisphere of $S^{n-s(k-1)-2}$ contains an 
$s$-stable $k$-subset of $[n]$ as desired. 
\end{proof}
For a graph $G$,  there are several topological spaces (simplicial complexes) 
associated with  $G$. In terms of some topological invariants of these spaces, 
there are several lower bounds for the chromatic number, 
see~\cite{MR514625,SiTaZs13,MR2279672,MR2351519}.  
For instance, the box-complex of $G$, denoted 
$B_0(G)$, is one of these spaces.   
Simonyi and  Tardos~\cite{MR2351519}, using Tucker-Bacon~lemma (Lemma~\ref{Bacon}), proved that if the chromatic 
number of $G$ equals to ${\rm coind}(B_0(G))+1$,  
then for any optimal coloring of $G$ with colors $[C]$ and for 
any partition $L\uplus M$ of $[C]$, there is a multi-colored complete bipartite subgraph 
$K_{|L|,|M|}$ of $G$ such that all colors in $L$ 
are assigned to the vertices of one side of $K_{|L|,|M|}$ and all colors in $M$ are 
assigned to the vertices of the other side. 
We should mentioned that there are some other similar statement, which are stated 
in terms of some other topological parameters related to the graph $G$. 
These kinds of results are known as 
$K_{l,m}$ type theorems, see~\cite{SiTaZs13,MR2351519}. 

\begin{lemma}{\rm (Tucker-Bacon lemma.)}~\label{Bacon}
Let $U_1,U_2,\ldots,U_{d+2}$ be open subsets of
the $d$-sphere $S^d$ such that for any $1\leq i\leq d+2$,
$U_i\cap -U_i=\varnothing$ and also, $U_1\cup\cdots\cup U_{d+2}=S^d$.
Then for any partition $A\cup B=\{1,2,\ldots,d+2\}$ for which $A\neq\varnothing$ and
$B\neq \varnothing$, there is an $x\in S^d$ such that $x\in \ds\bigcap_{i\in A}U_i$ and
$-x\in \ds\bigcap_{j\in B}U_j$.
\end{lemma} 

In what follows, similar to Simonyi and  Tardos's proof, 
we use Tucker-Bacon lemma to prove a  $K_{l,m}$ type theorem 
for $s$-stable Kneser graphs provided that $s$ is even. 
\begin{theorem}
Let $n,k$, and $s$ be positive integers, where $s$ is even and $n\geq(s+2)k-2.$ 
Also, let $c$ be a proper coloring of $\KG_s(n,k)$ with colors $\{1,2,\ldots,n-s(k-1)\}$ 
and assume that $A$ and $B$ form a partition of $\{1,2,\ldots,n-s(k-1)\}$ . 
Then there exists a complete bipartite subgraph $K_{l,m}$ of $\KG_s(n,k)$
with parts $L$ and $M$ such that $|L|=l$, $|M|=m$ and the vertices in $L$ and $M$ 
receive different colors from $A$ and $B$, respectively.
\end{theorem}
\begin{proof}   The idea of our proof is similar  to the B{\'a}r{\'a}ny's proof of Kneser conjecture~\cite{MR514626}.
By Lemma~\ref{galenew}, we can identify $[n]$ with a subset of $S^{n-s(k-1)-2}$ 
such that every open hemisphere of $S^{n-s(k-1)-2}$ contains an $s$-stable 
$k$-subset of $[n]$. 
For $1\leq i\leq {n-s(k-1)}$, define
 $$U_i\isdef\left\{x\in S^{n-s(k-1)-2}: H(x)\mbox{ contanis  a vertex with color } i \right\}.$$ 
 One can see that each $U_i$ is an open set, 
$U_1,U_2,\ldots,U_{n-s(k-1)}$ covers $S^{n-s(k-1)-2}$ and
also none of them contains a pair of antipodal points. Thus,
Tucker-Bacon~lemma implies that there is an $x\in S^{n-s(k-1)-2}$ such that
$x\in \ds\bigcap_{i\in A}U_i$ and $-x\in \ds\bigcap_{j\in B}U_j$. Therefore, 
in view of the definition of $U_i$'s, for each $i\in A$ (resp. $j\in B$), there is an $s$-stable 
$k$-subset $L_i$ (resp. $M_j$) of $[n]$ such that 
$c(L_i)=i$ and $L_i\subseteq H(x)$ (resp. $c(M_j)=j$ and $M_j\subseteq H(-x)$).
Note that since $H(x)\cap H(-x)=\varnothing$, for each $i\in A$ and $j\in B$, $L_i$ is 
adjacent to $M_j$ in $\KG_s(n,k)$, which completes the proof. 
\end{proof}


\section{Circular Coloring of Induced subgraphs of Kneser Graphs}
The concept of  free coloring  of graphs was introduced in \cite{MR2601263} by the first author and Hajiabolhassan as a tool for studying the circular chromatic number of graphs. 
Indeed,  they proved that if the free chromatic number  of a graph $G$ is at least twice its chromatic number, then $\chi(G)=\chi_c(G)$. 

An independent set in a graph $G$
is called a {\it free independent
set} if it can be extended to at least two distinct
maximal independent sets in $G$.
Clearly, one can see that an independent set $F$ in $G$ is a free
independent set
if and only if there exists an edge $uv\in E(G)$ such that
$(N(u)\cup N(v))\cap F=\varnothing$. The maximum possible size of a free independent set in $G$ is denoted by $\bar{\alpha}(G)$. 
Furthermore, a vertex of a graph $G$ is
contained in a free independent set if and only if the graph
obtained by deleting  the closed neighborhood of this vertex has at least one edge
(for more details,~see~\cite{MR2601263}).
As a natural extension of the chromatic number, we can define the free chromatic number of graphs as follows. 
\begin{definition}{
The {\it free chromatic number} of a graph $G$, denote $\phi(G)$, is the minimum size
of a partition of $V(G)$ into free independent sets. If $G$ does not have such a partition, then we set $\phi(G)=\infty$.}
\end{definition}
The next lemma has a key role in the rest of the paper.
\begin{lemma}{\rm \cite{MR2601263}}\label{freecircular}
Let $G$ be a graph such that $\chi_c(G)=\frac nd$ with ${\rm gcd}(n,d)=1$.
If $d\geq2$, or equivalently, if $\chi_c(G)\neq\chi(G)$, then
$\phi(G)\leq 2\chi(G)-1$.
\end{lemma}
The preceding lemma provides a sufficient condition for the
equality of the chromatic number and the circular chromatic number of
a graph $G$.

Let $G$ be a graph with at least one free independent set. 
It is easily seen that $\phi(G)\geq \frac {|V(G)|}{\bar{\alpha}(G)}$. 
It was proved by Hilton and Milner \cite{MR0219428} that if $T$
is an independent set of $\KG(n,k)$ of size at least 
$${n-1 \choose k-1}-{n-k-1 \choose k-1}+2,$$
then
$$\displaystyle \bigcap_{A \in T} A = \{i\},$$
for some $i \in [n]$. 
By using this result of Hilton and Milner, it  was proved in~\cite{MR1983360} that if 
$n\geq 2k^2(k-1)$, then $\chi_c(\KG(n,k)) =\chi(\KG(n,k))$. 
This result was improved in~\cite{MR2601263} by proving that we have 
$\chi_c(\KG(n,k)) =\chi(\KG(n,k))$ for $n\geq 2k^2(k-1)-2k+3$. 
It was also showed in~\cite{MR1983360} that there is a threshold $n(k)$ such that for $n\geq n(k)$, we have 
$\chi_c(\SG(n,k)) =\chi(\SG(n,k))$. This gave a positive answer to a question of Lih and 
Liu~\cite{Ko-We-2002}. Lih and Liu~\cite{Ko-We-2002} also posed the question of what 
the smallest value  of $n(k)$ is. They proved that $n(k)\geq 2k+2$. 
One should note that in~\cite{MR1983360} only the existence of the 
threshold $n(k)$ is ensured and the authors did~not present any upper bound for it.

Using Hilton and Milner's result, one can simply see that, for $n>2k$, 
the size of any free independent
 set in the Kneser graph $\KG(n,k)$ is at most
${n-1 \choose k-1}-{n-k-1 \choose k-1}\leq k{n-2 \choose k-2}$, 
see~\cite{MR2601263}. In view of this observation,  
we generalize the result
by Hajiabolhassan and Zhu~\cite{MR2601263} to the following theorem. 

\begin{theorem}
Let $n$ and $k$ be two positive integers, where $n\geq 2k^2(k-1)$. 
Let $H$  be an induced subgraph of $\KG(n,k)$ with at least $\frac{2k^2(k-1)}{n}{n\choose k}$ vertices.
 Then $H$ has the same  chromatic number and circular chromatic number.
\end{theorem}
\begin{proof}{
Obviously, the assertion holds for $k=1$. So, let $k\geq 2$.
Assume that $H$ is an induced subgraph of $\KG(n,k)$ with at least $\frac{2k^2(k-1)}{n}{n\choose k}$ vertices.
According to Lemma \ref{freecircular}, it is enough to show that 
 $\phi(H)\geq 2\chi(H)$. To this end, note that 
$$\begin{array}{rll}
  \phi(H) & \geq & {|V(H)|\over \bar{\alpha}(H)}\\ \\
          & \geq & {|V(H)|\over \bar{\alpha}(\KG(n,k))}\\ \\
          & \geq & { \frac{2k^2(k-1)}{n}{n\choose k}\over k{n-2 \choose k-2}}\\ \\
          & \geq & { 2k^2(k-1)n(n-1)\over nk^2(k-1)},
\end{array}$$
therefore $\phi(H)\geq 2n-2>2\chi(\KG(n,k))\geq2\chi(H)$ as desired.
}\end{proof}


In the rest of this section, we focus on the chromatic number and the circular chromatic number of the $s$-stable Kneser graph
$\KG_s(n,k)$, which is also an induced  subgraph of $\KG(n,k)$. 
As a special case of previous theorem, we introduce a threshold 
$n(k,s)$ such that for any $n\geq n(k,s)$, we have $\chi(\KG_s(n,k))=\chi_c(\KG_s(n,k))$.
In this regard, we first need to count the number of vertices of $\KG_s(n,k)$.

Let $N_i$ be the number of vertices of $\KG_s(n,k)$ containing $i$.
It is obvious that $N_i=N_j$ for all $i,j\in[n]$.
Also, let $A=\{x_1,\ldots,x_k\}$ be a vertex of $\KG_s(n,k)$,
where $1=x_1<x_2<\cdots<x_k\leq n$.
Define $y_i= x_{i+1}-x_{i}$ for all $1\leq i\leq k-1$ and $y_k= n-x_k+1$.
Since $A\in V(\KG_s(n,k))$ and $1\in A$, we have $y_i\geq s$ for all $i\in[k]$.
Also, since $y_1+y_2+\cdots+y_k=n$, any vertex $A$ of $\KG_s(n,k)$ with $1\in A$ leads us to a solution of the following system;
$$\left\{\begin{array}{l}
Z_1  + Z_2  +\cdots+  Z_k=n\\
Z_i\geq s\ \ \  {\rm for\ each\ } i\in[n]
 \end{array}\right.$$
and vise versa. Note that the number of the solutions of the preceding system is ${n-k(s-1)-1\choose k-1}$. Consequently, for each $i\in[n]$, we have 
$N_i=N_1={n-k(s-1)-1\choose k-1}$ for all $i\in [n]$.
By an easy double counting, one can see that
$$|V(\KG_s(n,k))|={1\over k}\displaystyle\sum_{i=1}^n N_i={n\over k}{n-k(s-1)-1\choose k-1}.$$

\begin{theorem}
If $n\geq 2k^2(k-1)+(s-1)k(k-1)+1$, then $\chi_c(\KG_s(n,k))=\chi(\KG_s(n,k))$.
\end{theorem}
\begin{proof}
Let $X$ be the number of $(k-1)$-subsets $B$ of the set $[n-1]$ such that $B\cap [(s-1)k]\neq \varnothing$, i.e., 
$$X=\#\left\{B\;:\; B\subseteq [n-1]\; \&\; B\cap [(s-1)k]\neq\varnothing\right\}.$$
Obviously, we have ${n-1\choose k-1}={n-(s-1)k-1\choose k-1}+X$. 
On the other hand, one can check that 
$X\leq (s-1)k{n-2 \choose k-2}$, which implies the following inequalities; 
$$\begin{array}{rll}
|V(\KG_s(n,k))| &  =   & \ds {n\over k}{n-k(s-1)-1\choose k-1} \\ \\
                      & \geq &  \ds{n\over k}\left({n-1\over k-1}-(s-1)k\right){n-2 \choose k-2}\\ \\
                      & \geq &  \ds{n\over k(k-1)}\left(n-1-(s-1)k(k-1)\right){n-2 \choose k-2}.
  \end{array}
$$
Clearly, the previous inequalities lead us to the following ones;
$$\begin{array}{rll}
\phi(\KG_s(n,k)) & \geq & {|V(\KG_s(n,k))|\over \bar{\alpha}(\KG_s(n,k))}\\ \\
                       & \geq & {{n\over k(k-1)}\left(n-1-(s-1)k(k-1)\right){n-2 \choose k-2}\over k{n-2 \choose k-2}}\\ \\
                       & \geq & {n\over k^2(k-1)}\left(n-1-(s-1)k(k-1)\right)
\end{array}.$$
Consequently, we have 
$\phi(\KG_s(n,k))\geq 2n \geq 2(n-s(k-1))$ provided that $n\geq 2k^2(k-1)+(s-1)k(k-1)+1.$
Considering Lemma \ref{freecircular}, the proof is completed. 
\end{proof}

Note that for $s=2$, previous theorem gives an upper bound for the smallest value of the  threshold $n(k)$, giving a partial answer to the question posed by Lih and 
Liu~\cite{Ko-We-2002}.

\section*{Acknowledgment} We would like to acknowledge Professor Hossein Hajiabolhassan for his invaluable comments and suggestions.

\bibliographystyle{plain}

\def\cprime{$'$} \def\cprime{$'$}

\end{document}